\documentclass[12pt]{amsart}
\usepackage{amsmath,amsthm,amsfonts,amscd,amssymb,eucal,latexsym,mathrsfs, appendix, yhmath, setspace}
\usepackage{color,colortbl}
\usepackage[numbers,sort&compress]{natbib}
\usepackage[all,cmtip]{xy}
\usepackage{enumerate}

\newtheorem{theorem}{Theorem}[section]
\newtheorem{corollary}[theorem]{Corollary}
\newtheorem{lemma}[theorem]{Lemma}
\newtheorem{proposition}[theorem]{Proposition}

\theoremstyle{definition}
\newtheorem{definition}[theorem]{Definition}
\newtheorem{remark}[theorem]{Remark}

\newtheorem{example}[theorem]{Example}
\newtheorem{question}[theorem]{Question}

\DeclareMathOperator{\Ima}{im \ }
 \newcommand{\Ker}{{\rm ker \ }}
  \newcommand{\Coker}{{\rm coker \ }}

\newcommand{\Hom}{{\rm Hom}}
 \newcommand{\Map}{{\rm Map}}
\newcommand{\rk}{{\rm rk}}
\newcommand{\mrk}{{\rm mrk}}
\newcommand{\mrks}{{\rm mrk_\Sigma}}

\newcommand{\vrk}{{\rm vrk}}
\newcommand{\mdim}{{\rm mdim}}

\newcommand{\mdims}{{\rm mdim_{\Sigma}}}

\newcommand{\rL}{{\rm L}}
\newcommand{\Sym}{{\rm Sym}}
\newcommand{\tr}{{\rm tr}}

  \newcommand{\cC}{{\mathcal C}}
  \newcommand{\cD}{{\mathcal D}}
  \newcommand{\cF}{{\mathcal F}}
  \newcommand{\cJ}{{\mathcal J}}
  \newcommand{\cL}{{\mathcal L}}
 \newcommand{\cM}{{\mathcal M}}
 \newcommand{\cN}{{\mathcal N}}

 \newcommand {\cU}{{\mathcal U}}
  \newcommand {\cV}{{\mathcal V}}

  \newcommand{\bC}{{\mathbb C}}
  \newcommand{\bF}{{\mathbb F}}
 \newcommand{\bN}{{\mathbb N}}
 \newcommand{\bP}{{\mathbb P}}
 
 \newcommand{\bR}{{\mathbb R}}
 
 \newcommand{\bZ}{{\mathbb Z}}
  \newcommand{\CG}{{\mathbb C \Gamma}}
 \newcommand{\ZG}{{\mathbb Z \Gamma}}
\newcommand{\LG}{{\mathcal L \Gamma }}
 \newcommand{\hcM}{\widehat{\mathcal M}}

 \newcommand{\sA}{{\mathscr A}}
\newcommand{\sB}{{\mathscr B}}

\newcommand{\sF}{{\mathscr F}}

\newcommand{\sM}{{\mathscr M}}

\allowdisplaybreaks

\begin{document}

\title{Dynamical correspondences of  $L^2$-Betti numbers}

\author{Bingbing Liang}

\address{\hskip-\parindent
B.L., Max Planck Institute for Mathematics, Vivatsgasse 7, 53111, Bonn, Germany }
\email{bliang@mpim-bonn.mpg.de}


\date{March 1, 2017}

\begin{abstract}
We investigate dynamical analogues of the $L^2$-Betti numbers for modules over integral group ring of a discrete sofic group.  In particular, we show that the $L^2$-Betti numbers exactly measure the failure of addition formula for dynamical invariants.
\end{abstract}

\maketitle

\tableofcontents

%%%%%%%%%%%%%%%%%%%%%%%%%%%%%%%%%%%%%%%%%%%%%%%%%%%%%%%%%%%%%%%%%%%%%%%%%%%%%%%%%%%%%%%%%
\section{Introduction} \label{S-introduction}

There are a couple of connections established among invariants in dynamical systems, group rings, and $L^2$-invariants. These connections are obtained via a type of dynamical system called algebraic actions. Given a discrete group $\Gamma$, each $\ZG$-module $\cM$ can be treated as an action of $\Gamma$ on the discrete abelian group $\cM$ by group automorphisms. The Pontryagin dual $\hcM$ of $\cM$ naturally inherits an action of $\Gamma$  by continuous automorphisms from the module structure of $\cM$. Conversely, by Pontryagin duality, each action of $\Gamma$ on a compact Hausdorff abelian group arise this way and thus we call such a dynamical system an {\it ``algebraic action"} \cite{Schmidt95}.

A surprising fact is that one can recover certain algebraic information about $\cM$ by taking advantage of purely dynamical information about $\Gamma \curvearrowright \hcM$. However, the dynamical information itself does not use the algebraic structure of $\hcM$. For example, Li and Thom showed that, in the setting of amenable group actions, the entropy of $\Gamma \curvearrowright \hcM$ coincides with the $L^2$-torsion of $\cM$ (see \cite{LT14}). One ingredient of establishing this connection is Peters' algebraic characterization of entropy \cite{Peters79}. This correspondence has interesting applications to the vanishing results on $L^2$-torsion and Euler characteristic \cite{LT14, CT15}. In the same spirit, Li and the author showed that the mean topological dimension of $\Gamma \curvearrowright \hcM$ coincides with the von Neumann-L\"{u}ck rank of $\cM$ (see \cite{LL13}). Establishing this correspondence relies on the study of the mean rank as an algebraic invariant of $\ZG$-modules and L\"{u}ck's result on dimension-flatness for amenable groups \cite[Theorem 6.73]{Lueck02B}. 
Based on this connection, the mean dimension of algebraic actions for amenable groups is well understood \cite{LL13}.

Mean topological dimension is a newly-introduced dynamical invariant by Gromov \cite{Gromov99M}, systematically studied by Lindenstrauss and Weiss \cite{LW00}, and remains to be further explored \cite{Coornaert15B}. As a dynamical analogue of the covering dimension, it is closely related to the topological entropy, and takes a crucial role in embedding problem of dynamical systems \cite{Gutman11, Gutman15, Gutman16, GLT16, GT14, GT15A, LW00}.

On the other hand, using L\"{u}ck's extended von Neumann dimension for any module over the group von Neumann algebra $\LG$ of a discrete group $\Gamma$ (see \cite[Chapter 6]{Lueck02B}), for any $\ZG$-module $\cM$ we call the von Neumann-L\"{u}ck dimension for $\LG \otimes_\ZG \cM$ as the {\it von Neumann-L\"{u}ck rank}  $\vrk(\cM)$ of $\cM$. Von Neumann-L\"{u}ck dimension is a length function on $\LG$-modules \cite[Theorem 6.7]{Lueck02B} and von Neumann-L\"{u}ck rank is a length function on $\ZG$-modules when $\Gamma$ is amenable \cite[Definition 2.1]{LL15A} \cite[Section 5.2]{LL13} \cite[Theorem 3.3.4]{Liang16D}.

Mean rank is also a length function on $\ZG$-modules of an amenable group $\Gamma$ (see \cite[Section 3]{LL13}). As a dynamical analogue of the rank of abelian groups, it serves as a bridge connecting mean dimension and von Neumann-L\"{u}ck rank \cite[Theorem 1.1]{LL13}.

Towards more general groups, Bowen and Kerr-Li developed an entropy theory based on the idea of approximating the dynamical data by external finite models when the acting group can be approximated by finite groups \cite{Bowen10, KL11}. The groups admitting this approximation are the so-called {\it sofic groups}, which include residually finite groups and amenable groups \cite{Gromov99S, Weiss99}. The extended notion of entropy extends the classic notion but no longer decreases when passing to a factor system. Similarly mean dimension has been extended to the case of sofic group actions \cite{Li13}. To deal with this nonamenable phenomenon, Li and the author introduced the relative sofic invariants, established an alternative addition formula, and used them to relate mean dimension with von Neumann-L\"{u}ck rank for sofic groups \cite{LL15A}. Similar approaches also independently appears in the works of other experts. Hayes gave a formula for this invariant in terms of a given compact model in \cite{Hayes15}. A similar notion for Rokhlin entropy, called outer Rokhlin entropy, was developed by Seward in  \cite{SewardII}. Using the microstate technique, Hayes proved that von Neumann-L\"{u}ck rank of a finitely presented $\ZG$-module $\cM$ coincides with sofic mean dimension of $\Gamma \curvearrowright \hcM$ under certain conditions \cite{Hayes13A}. 

Via a projective resolution of any $\ZG$-module $\cM$, we can treat von Neumann-L\"{u}ck rank of $\cM$ as the $0$-th $L^2$-Betti number $\beta_0^{(2)}(\cM)$ of $\cM$ (see Proposition \ref{0th case}). From \cite[Theorem 1.3]{LL15A}, we know the sofic mean dimension $\mdims(\hcM)$ of $\Gamma \curvearrowright \hcM$ correspondences to $\beta_0^{(2)}(\cM)$ when $\Gamma$ is a countable sofic group and $\cM$ is countable. Here $\Sigma$ is a fixed sofic approximation sequence for $\Gamma$.  For the higher $L^2$-Betti numbers of $\cM$, Hanfeng Li asked the following question.

\begin{question} \label{question}
If $\Gamma$ is sofic, what dynamical invariants of $\Gamma \curvearrowright \hcM$ correspond to the $j$-th $L^2$-Betti numbers $\beta_j^{(2)}(\cM)$ of $\cM$ for $j \geq 1$?
\end{question}

In this paper, motivated by the above question, we mainly study dynamical analogues of the $L^2$-Betti numbers $\beta_j^{(2)}(\cC_\ast)$ of a chain complex $\cC_\ast$ of $\ZG$-modules:
$$\cdots \stackrel{\partial_2}{\to} C_1 \stackrel{\partial_1}{\to} C_0 \to 0(=C_{-1}).$$
In the spirit of Elek \cite{Elek02} (also for the notational convenience), we introduce the $j$-th mean rank  $\mrk_j(\cC_\ast)$ of $\cC_\ast$ and the $j$-th mean dimension $\mdim_j(\cC_\ast)$ of $\widehat{\cC_\ast}:=\Hom_\bZ(\cC_\ast, \bR/\bZ)$ for any sofic group $\Gamma$ (see Definition \ref{mean rank for chain complex}). These definitions use the relative sofic invariants as opposed to Elek's approach where he considered the case that  $\Gamma$ is amenable and therefore there is no nonamenable phenomenon appeared.

Let $\Gamma \curvearrowright X$ and $\Gamma \curvearrowright Y$ be two algebraic actions, $X$ and $Y$ be metrizable spaces, and $\pi: X \to Y$ be a $\Gamma$-equivariant continuous homomorphism.  We say $\pi$  satisfies {\it Juzvinski\u{\i} formula for mean dimension} if $\mdims(X)=\mdims(\Ker \pi)+\mdims(\Ima \pi)$. The main result of this paper is as follows.
\begin{theorem} \label{theorem for correspondence}
Suppose that  $\vrk(C_j) <\infty$ for some $j\geq 0$.  Then
\begin{enumerate}
\item 
$\beta_j^{(2)}(\cC_\ast)=\vrk(\Coker \partial_{j+1})-\vrk(\Ima \partial_j|C_{j-1});$
\item 
If  $\Gamma$ is sofic, we have
$\beta_j^{(2)}(\cC_\ast)=\mrk_j(\cC_\ast)$.
If furthermore $C_j$ and $C_{j-1}$ are countable, we have
$\mrk_j(\cC_\ast)=\mdim_j(\widehat{\cC_\ast})$;
\item 
If $\Ima \partial_{j+1}=\Ker \partial_j$, we have that $\beta_j^{(2)}(\cC_\ast)=0$ if and only if
$$\vrk(C_{j-1})=\vrk(\Coker \partial_j)+\vrk(\Ima \partial_j).$$
If furthermore $\Gamma$ is sofic and $C_j$ and $C_{j-1}$ are countable, we have that $\beta_j^{(2)}(\cC_\ast)=0$ if and only if $\widehat{\partial_j}$ satisfies \it Juzvinski\u{\i} formula for mean dimension.
\end{enumerate}
\end{theorem}

From Theorem \ref{theorem for correspondence}, the notion of $j$-th mean rank provides an equivalent algebraic definition of $L^2$-Betti numbers from module theory. Secondly, the $L^2$-Betti numbers exactly measure the failure of the  additivity of dynamical invariants. In \cite{Elek02}, Elek introduced an analogue of the $L^2$-Betti numbers for amenable linear subshifts. It was showed that  Juzvinski\u{\i} formula for entropy  can fail when the group $\Gamma$ has nonzero Euler characteristic \cite{Elek99}. Hayes proved that Juzvinski\u{\i} formula for entropy fails when $\Gamma$ has nonzero $L^2$-torsion \cite{Hayes16}. Gaboriau and Seward established some inequalities relating Juzvinski\u{\i} formula for entropy with $L^2$-Betti numbers \cite{GS15A}. Bowen and Gutman established Juzvinski\u{\i} formula for the  $f$-invariant of finitely generated free group actions \cite{BG14}.

To respond to Question \ref{question}, we introduce the $j$-th mean dimension $\mdim_j(\hcM)$ of $\Gamma \curvearrowright \hcM$ and $j$-th mean rank $\mrk_j(\cM)$ of $\cM$ (Definition \ref{j-th mdim} and Definition \ref{mean rank for chain complex}). As the first application, the following corollary may shed some light on Question \ref{question}.

\begin{corollary} \label{main theorem}
When $\mrk_j(\cM)$ is defined, 
we have $\mrk_j(\cM)=\beta_j^{(2)}(\cM)$.
If furthermore $\cM$ is countable, we have $\mdim_j(\widehat{\cM})=\beta_j^{(2)}(\cM).$
\end{corollary}

For the second application, we give a dynamical characterization of L\"{u}ck's dimension-flatness. We say $\Gamma$ satisfies {\it L\"{u}ck's dimension-flatness over $\bZ$} if $\beta_j^{(2)}(\cM)$ vanishes for any $j \geq 1$ and $\ZG$-module $\cM$. It was proven that amenable groups satisfy L\"{u}ck's dimension-flatness \cite[Theorem 6.37]{Lueck02B}. We say $\Gamma$ satisfies {\it Juzvinski\u{\i} formula for  von Neumann-L\"{u}ck rank} if $\vrk(\cM)=\vrk(\Ker \varphi)+\vrk(\Ima \varphi)$ for any $\ZG$-module homomorphism $\varphi: \cM \to \cN$ of $\ZG$-modules $\cM$ and $\cN$. It is similarly defined when we talk about whether $\Gamma$ satisfies Juzvinski\u{\i} formula for mean rank.

\begin{corollary} \label{dimension-flat}
$\Gamma$ satisfies {\it L\"{u}ck's dimension-flatness over $\bZ$} if and only if $\Gamma$ satisfies  Juzvinski\u{\i} formula for von Neumann-L\"{u}ck rank.  If $\Gamma$ is sofic, then $\Gamma$ satisfies {\it L\"{u}ck's dimension-flatness over $\bZ$} if and only if $\Gamma$ satisfies  Juzvinski\u{\i} formula for mean rank and mean dimension.
\end{corollary}

We remark that the first statement of the above corollary can also be proved using standard properties of Tor functor and additivity of von Neumann-L\"{u}ck dimension. In the light of results on the failure of Juzvinski\u{\i} formula  \cite[Proposition 7.2]{Hayes13A} \cite[Corollary 6.24]{Hayes16} \cite[Theorem 6.3]{GS15A}, we show that taking subgroups respects the property of L\"{u}ck's dimension-flatness in Proposition \ref{subgroup}. As a consequence, if $\beta_j^{(2)}(H) > 0$ for some subgroup $H$ of $\Gamma$ and some $j\geq 1$, then $\Gamma$ violoates Juzvinski\u{\i} formula for mean dimension.

L\"{u}ck conjectured that a group is amenable if and only if it satisfies {\it L\"{u}ck's dimension-flatness} \cite[Conjecture 6.48]{Lueck02B}. Bartholdi and Kielak implicitly proved this conjecture using a new characterization of amenability \cite[Theorem 1.1]{BK16}.  It follows that

\begin{corollary}
A countable group is amenable if and only if it satisfies  Juzvinski\u{\i} formula for von Neumann-L\"{u}ck rank.
\end{corollary}

This paper is organized as follows. We recall some background knowledge in Section 2. In Section 3 we introduce the $j$-th mean rank, $j$-th mean dimension, and establish some basic properties. We prove the main results and show some applications in Section 4.

Throughout this paper, $\Gamma$ will be a countable discrete group. For any set $S$, we denote by $\cF(S)$ the set of all nonempty finite subsets of $S$. All modules are assumed to be left modules unless specified. For any $d \in \bN$, we write $[d]$ for the set $\{1, \cdots, d\}$ and $\Sym(d)$ for the permutation group of $[d]$.

\noindent{\it Acknowledgements.}
We are grateful to Lewis Bowen, Ben Hayes, Fabian Henneke, Yonatan Gutman, Huichi Huang, Yang Liu, Yongle Jiang, Wolfgang L\"{u}ck, Jianchao Wu, and Xiaolei Wu for helpful discussions and comments.  The author is  supported by  Max Planck Institute for Mathematics in Bonn.

%%%%%%%%%%%%%%%%%%%%%%%%%%%%%%%%%%%%%%%%%%%%%%%%%%%%%%%%%%%%%%%%%%%%%%%%%%%%%%%%%%%%%%%%%%
\section{Preliminaries} \label{S-preliminary}

%%%%%%%%%%%%%%%%%%%%%%%%%%%%%%%%%%%%%%%%%%%%%%%%%%%%%%%%%%%%%%%
\subsection{Group rings} 

The {\it integral group ring of $\Gamma$}, denoted by $\ZG$, consists of all finitely supported functions $f: \Gamma\rightarrow \bZ$. We shall write $f$ as $\sum_{s\in \Gamma}f_ss$, where $f_s\in \bZ$ for all $s\in \Gamma$ and $f_s=0$ for all except finitely many $s\in \Gamma$. The algebraic operations on $\ZG$ are defined by
$$ \sum_{s\in \Gamma}f_ss+\sum_{s\in \Gamma}g_ss=\sum_{s\in \Gamma}(f_s+g_s)s, \mbox{ and } \big(\sum_{s\in \Gamma}f_s s\big)\big(\sum_{t\in \Gamma}g_tt\big)=\sum_{s, t\in \Gamma}f_sg_t(st).$$
We similarly have the product if one of $f$ and $g$ sits in $\bC^\Gamma$.

For any countable $\ZG$-module $\cM$, treated as a discrete abelian group, its Pontryagin dual $\hcM$ consisting of all continuous group homomorphisms $\cM \to \bR/\bZ$,  coincides with ${\rm Hom}_{\bZ}(\cM, \bR/\bZ)$. By Pontryagin duality, $\hcM$ is a compact metrizable space under compact-open topology. Furthermore, the $\ZG$-module structure of $\cM$ naturally induces an adjoint action $\Gamma \curvearrowright \hcM$ by continuous automorphisms. To be precise,
$$\langle s\chi, u\rangle: =\langle \chi, s^{-1}x \rangle$$
for all $\chi \in \hcM, u \in \cM$, and $s \in \Gamma$.

%%%%%%%%%%%%%%%%%%%%%%%%%%%%%%%%%%%%%%%%%%%%%%%%%%%%%%%%%%%%%%%%%
\subsection{Relative von Neuman-L\"{u}ck rank }

Let $\ell^2(\Gamma)$ be the Hilbert space of square summable functions $f: \Gamma\rightarrow \bC$, i.e. $\sum_{s\in \Gamma}|f_s|^2<+\infty$. Then $\Gamma$ has two canonical commuting unitary representations on $\ell^2(\Gamma)$, namely the {\it left regular representation} $\lambda$ and the {\it right regular representation} $\rho$ defined by
$$ \lambda(s)(x)=sx, \mbox{ and } \rho(s)(x)=xs^{-1}$$
for all $x\in \ell^2(\Gamma)$ and $s\in \Gamma$. Here we treat $\Gamma$ as a subset of $\CG$.
The {\it (left) group von Neumann algebra} of $\Gamma$, denoted by $\cL\Gamma$, consists of all bounded linear operators
$\ell^2(\Gamma)\rightarrow \ell^2(\Gamma)$ commuting with $\rho(s)$ for all $s\in \Gamma$. 

Denote by $\delta_{e_\Gamma}$ the unit vector of $\ell^2(\Gamma)$ being $1$ at the identity element $e_\Gamma$ of $\Gamma$, and $0$ everywhere else. The canonical {\it trace} on $\cL\Gamma$ is the linear functional $\tr_{\cL\Gamma}: \cL\Gamma\rightarrow \bC$ given by $\tr_{\cL\Gamma}(T)=\left<T\delta_{e_\Gamma}, \delta_{e_\Gamma}\right>$.
 For each $n\in \bN$, the extension of $\tr_{\cL\Gamma}$ to $M_n(\cL\Gamma)$ sending $(T_{j, k})_{1\leq j, k\leq n}$ to $\sum_{j=1}^n\tr_{\cL\Gamma}(T_{j, j})$
 will still be denoted by $\tr_{\cL\Gamma}$. 

For any finitely generated projective $\cL\Gamma$-module $\bP$, one has $\bP\cong (\cL\Gamma)^{1\times n}P$ for some $n\in \bN$ and some $P\in M_n(\cL\Gamma)$ with $P^2=P$. The {\it von Neumann dimension}  of $\bP$ is defined as
$$ \dim'_{\cL\Gamma}(\bP):=\tr_{\cL\Gamma}(P)\in [0, n],$$
which does not depend on the choice of $n$ and $P$. For an arbitrary $\cL\Gamma$-module $\mathbb{M}$, its {\it von Neumann-L\"{u}ck dimension} \cite[Definition 6.6]{Lueck02B} is defined as
$$ \dim_{\cL\Gamma}(\mathbb{M}):=\sup_{\bP}\dim'_{\cL\Gamma}(\bP),$$
for $\bP$ ranging over all finitely generated projective $\cL\Gamma$-submodules of $\mathbb{M}$.

The following theorem collects the fundamental properties of the von Neumann-L\"{u}ck dimension \cite[Theorem 6.7]{Lueck02B}. Given a unital ring $R$, a length function on left $R \Gamma$-modules is a function on left $R \Gamma$-modules satisfying certain conditions (\cite[Definition 2.1]{LL15A}).

\begin{theorem} \label{T-vdim}
$\dim_{\cL\Gamma}$ extends $\dim'_{\cL \Gamma}$ and is a length function on $\cL\Gamma$-modules with $\dim_{\cL\Gamma}(\cL\Gamma)=1$.
\end{theorem}

\begin{definition}
For any $\ZG$-modules $\cM_1 \subseteq \cM_2$, the {\it von Neumann-L\"{u}ck rank of $\cM_1$ relative to $\cM_2$} is defined as
$$\vrk(\cM_1|\cM_2):=\dim_\LG (\Ima 1\otimes i),$$
where $1\otimes i$ is the natural map $\LG \otimes \cM_1 \to \LG \otimes \cM_2$.
\end{definition}

Note that when $\cM_1=\cM_2$, we have $\vrk(\cM_1)=\vrk(\cM_1|\cM_2)$.
%%%%%%%%%%%%%%%%%%%%%%%%%%%%%%%%%%%%%%%%%%%%%%%%%%%%%%%%%%%%%%%%%%
\subsection{Amenable and sofic groups }

The group $\Gamma$ is called {\it amenable} if for any $K\in \cF(\Gamma)$ and any $\delta>0$ there is a $F\in \cF(\Gamma)$ with $|KF\setminus F|<\delta |F|$.

A sequence of maps $\Sigma=\{\sigma_i: \Gamma \to {\rm \Sym} (d_i)\}_{i \in \bN}$ is called a {\it sofic approximation} for $\Gamma$  if it satisfies:
\begin{enumerate}
\item $\lim_{i\to \infty}|\{v\in [d_i]: \sigma_{i,s}\sigma_{i,t}(v)=\sigma_{i, st}(v)\}|/d_i=1$ for all $s, t\in \Gamma$,

\item $\lim_{i\to \infty}|\{v\in [d_i]: \sigma_{i, s}(v)\neq \sigma_{i,t}(v)\}|/d_i=1$ for all distinct $s, t\in \Gamma$,

\item $\lim_{i\to \infty} d_i=+\infty$.
\end{enumerate}
The group $\Gamma$ is called a {\it sofic group} if it admits a sofic approximation.

Any amenable group is sofic since one can use a sequence of asymptotically-invariant subsets of the amenable group, i.e. {\it F{\o}lner sequence},  to construct a sofic approximation. Residually finite groups are also sofic since a sequence of exhausting finite-index subgroups naturally induces a sofic approximation in which each approximating map is actually a group homomorphism. We refer the reader to \cite{CL15B, CC10B} for more information on sofic groups.

Throughout the rest of this paper, $\Gamma$ will be a countable sofic group, and $\Sigma=\{\sigma_i: \Gamma \to {\rm \Sym} (d_i)\}_{i \in \bN}$ will be a sofic approximation for $\Gamma$.

\subsection{Relative mean dimension and relative mean rank}

We first recall the notion of the covering dimension.
For any finite open cover $\cU$ of a compact metrizable space $Z$, denote the overlapping number of $\cU$ by ${\rm ord}(\cU)$, i.e. ${\rm ord}(\cU)=\max_{x\in X} \sum_{U \in \cU} 1_U(x)-1$. Set 
$$\cD(\cU)=\inf_{\cV} {\rm ord}(\cV)$$
for $\cV$ ranging over all finite open covers of $Z$ finer than $\cU$, i.e. each element of $\cV$ is contained in some element of $\cU$. Then the {\it covering dimension} of $Z$ is defined as $\sup_\cU \cD(\cU)$ for $\cU$ ranging over all finite open covers of $Z$.

Let $\Gamma$ act continuously on a compact metrizable space $X$.

\begin{definition}
Let $\rho$ be a compatible metric on $X$. For any $d \in \bN$, there is a compatible  metric on $X^d$ defined by
$$\rho_2(\varphi, \psi)=\left(\frac{1}{d}\sum_{v \in [d]} \rho(\varphi_v, \psi_v)^2\right)^{1/2}.$$
Let $\sigma$ be a map from $\Gamma$ to $\Sym(d)$, $F \in \cF(\Gamma)$, and $\delta > 0$. The set of approximately equivariant maps $\Map(\rho, F, \delta, \sigma)$ is defined to be the set of all maps $\varphi: [d] \to X$ such that $\rho_2(s\varphi, \varphi \circ \sigma(s)) \leq \delta$.
\end{definition}

Now let $\Gamma$ act on another compact metrizable space $Y$ and $\pi: X \to Y$ be a surjective  $\Gamma$-equivariant continuous map. Denote by $\Map(\pi, \rho, F, \delta, \sigma)$ the set of all $\pi \circ \varphi$ for $\varphi$ ranging in $\Map(\rho, F, \delta, \sigma)$. Note that $\Map (\pi,\rho, F, \delta, \sigma)$ is a closed subset  of $Y^d$. For any finite open cover $\cU$ of $Y$, denote by $\cU^d$ the open cover of $Y^d$ consisting of $\Pi_{v \in [d]} U_v$, where each $U_v$ sits in $\cU$.  Restricting $\cU^d$ to  $\Map(\pi, \rho, F, \delta, \sigma)$, we obtain a finite open cover $\cU^d|_{\Map(\pi, \rho, F, \delta, \sigma)}:=\{U \cap \Map(\pi, \rho, F, \delta, \sigma)\}_{U \in \cU^d}$ of $\Map(\pi, \rho, F, \delta, \sigma)$. 

\begin{definition}
 For any finite open cover $\cU$ of $Y$ we define
$$\mdims(\pi, \cU, \rho, F, \delta)=\varlimsup_{i \to \infty} \frac{\cD(\cU^d|_{\Map(\pi, \rho, F, \delta, \sigma_i)})}{d_i}.$$
If $\Map(\rho, F, \delta, \sigma_i)$ is empty for all sufficiently large $i$, we set  $\mdims(\pi, \cU, \rho, F, \delta)=-\infty$. We define the {\it mean topological dimension of $\Gamma \curvearrowright Y$ relative to the extension $\Gamma \curvearrowright X$ } as
$$\mdims(Y|X):=\sup_\cU \inf_{F \in \cF(\Gamma)} \inf_{\delta >0}\mdims(\pi, \cU, \rho, F, \delta),$$
where $\cU$ ranges over finite open covers of $Y$.
By a similar argument as in \cite[Lemma 2.9]{Li13}, we know $\mdims(Y|X)$ does not depend on the choice of $\rho$. The {\it sofc mean topological dimension of $\Gamma \curvearrowright X$} is defined as 
$$\mdims(X):=\mdims(X|X)$$ for $\pi: X \to X$ being the identity map.
\end{definition}

\begin{example}
Let $\cM_1 \subseteq \cM_2$ be countable $\ZG$-modules. Then the induced map $\widehat{\cM_2} \to \widehat{\cM_1}$ is a surjective $\Gamma$-equivariant continuous map of compact metrizable spaces. Thus $\mdims(\widehat{\cM_1}|\widehat{\cM_2})$ is well-defined.
\end{example}

Now we recall the notion of the relative mean rank. For any $\ZG$-module $\cM$, denote by $\sF(\cM)$ the set of finitely generated abelian subgroups of $\cM$. Let $\sA, \sB \in \sF(\cM), F \in \cF(\Gamma)$, and $\sigma$ be a map from $\Gamma$ to $\Sym(d)$ for some $d \in \bN$. Denote by $\sM(\sA, \sB, F, \sigma)$ the image of $\sA^d$ in $\cM^d/\sM(\sB, F, \sigma)$ under the quotient map $\cM^d \to \cM^d/\sM(\sB, F, \sigma)$. Here $\sM(\sB, F, \sigma)$ denotes the abelian subgroup of $\cM^d\cong \bZ^d \otimes_\bZ \cM$ generated by the elements $\delta_v \otimes b -\delta_{sv} \otimes sb$ for all $v \in [d], b \in \sB$, and $s \in F$.

\begin{definition}
Let $\cM_1 \subseteq \cM_2$ be $\ZG$-modules. For any $\sA \in \sF(\cM_1), \sB \in \sF(\cM_2)$, and $F \in \cF(\Gamma)$, set
$$\mrks(\sA|\sB, F)=\varlimsup_{i \to \infty } \frac{\rk(\sM(\sA, \sB, F, \sigma_i))}{d_i}.$$
We define the {\it mean rank of $\cM_1$ relative to $\cM_2$} as
$$\mrks(\cM_1|\cM_2) =\sup_{\sA \in \sF(\cM_1)} \inf_{F \in \cF(\Gamma)} \inf_{\sB \in \sF(\cM_2)} \mrks(\sA|\sB, F).$$
The {\it sofic mean rank } of $\cM_1$ is then defined as 
$$\mrks(\cM_1):=\mrks(\cM_1|\cM_1).$$
\end{definition}

Applying \cite[Theorem 1.1]{LL15A},\cite[Theorem 7.2]{LL15A}, \cite[Theorem 10.1]{LL15A}, and running a similar argument as in the proof of \cite[Proposition 8.5]{LL15A} for relative mean rank, we have:

\begin{theorem} \label{addition formula for mrk}
For any $\ZG$-modules $\cM_1 \subseteq \cM_2$, we have $\mrks(\cM_1|\cM_2)=\vrk(\cM_1|\cM_2)$ and 
$$\mrks(\cM_2)=\mrks(\cM_1|\cM_2)+\mrks(\cM_2/\cM_1).$$
If furthermore $\cM_2$ is countable, we have $\mdims(\widehat{\cM_1}|\widehat{\cM_2})=\mrks(\cM_1|\cM_2)$.
\end{theorem}

The following proposition collects basic properties of the sofic mean rank \cite[Section 3]{LL15A}.
\begin{proposition} \label{properties of mrk}
Let $\cM_1 $ and $\cM_2$ be $\ZG$-modules. The following are true.
\begin{enumerate}
\item $\mrks(\ZG)=1$.
\item $\mrks(\cM_1|\cM_1 \oplus \cM_2)=\mrks(\cM_1)$ and $\mrks(\cM_1 \oplus \cM_2)=\mrks(\cM_1)+\mrks(\cM_2).$
\item If $\cM_1 \subseteq \cM_2$ and $\cM_1$ is the union of an increasing net of $\ZG$-submodules $\{\cM_j'\}_{j \in \cJ}$, then $\mrks(\cM_j'|\cM_2) \nearrow \mrks(\cM_1|\cM_2)$. If furthermore $\mrks(\cM_2) < \infty$, then $\mrks(\cM_2/\cM_j') \searrow \mrks(\cM_2/\cM_1) $.
\item Assume that $\cM_1 \subseteq \cM_2$, $\cM_1$ is finitely generated, and $\cM_2$ is the union of an increasing net of $\ZG$-submodules $\{\cM_j'\}_{j \in \cJ}$ of $\cM_2$ containing $\cM_1$. Then $\mrks(\cM_1|\cM_j') \searrow \mrks(\cM_1|\cM_2)$.
\end{enumerate}
\end{proposition}

%%%%%%%%%%%%%%%%%%%%%%%%%%%%%%%%%%%%%%%%%%%%%%%%%%%%%%%%%%%%%%%%%%%%%%%%%%%%%%%%%%%%%%%%%%%
\section{$L^2$-Betti number, $j$-th mean rank, and $j$-th mean dimension}

Let $\cC_\ast$ be a chain complex  of $\ZG$-modules: 
$$\cdots \stackrel{\partial_2}{\to} C_1 \stackrel{\partial_1}{\to} C_0 \stackrel{\partial_0}{\to} 0  (=C_{-1}).$$
Applying the covariant tensor functor $\LG \otimes_{\ZG}\cdot$, we get a chain complex $\LG \otimes_\ZG \cC_\ast$ of $\LG$-modules: 
$$\cdots \stackrel{1\otimes \partial_2}{\longrightarrow} \LG \otimes C_1 \stackrel{1\otimes \partial_1}{\longrightarrow} \LG \otimes C_0 \to 0;$$
applying the contravariant Pontryagin dual functor ${\rm Hom}_\bZ(\cdot, \bR/\bZ):=\widehat{\cdot}$, we get a chain complex $\widehat{\cC_\ast}$ of 
algebraic actions such that the maps $\{\widehat{\partial_j}\}_j$ are $\Gamma$-equivariant:
$$\cdots  \stackrel{\widehat{\partial_2}}{\longleftarrow}\widehat{C_1} \stackrel{\widehat{\partial_1}}{\longleftarrow} \widehat{C_0} \longleftarrow 0.$$
Since $\bR/\bZ$ is an injective $\bZ$-module, when $\Ima \partial_{j+1}=\Ker \partial_j$, we have $\Ker \widehat{\partial_{j+1}}=\Ima \widehat{\partial_j}$.

\begin{definition} \label{mean rank for chain complex}
 For each $j\geq 0$, the {\it j-th $L^2$-Betti number of $\cC_\ast$} is defined as 
$$\beta_j^{(2)}(\cC_\ast)=\dim_\LG H_j(\LG \otimes_\ZG \cC_\ast).$$
If $\vrk(C_j) < \infty$ for all $j \geq 1$ and $C_j=0$ as $j $ is large enough, we define the {\it Euler characteristic} of $\cC_\ast$ as 
$$\chi(\cC_\ast):=\sum_{j \geq 0} (-1)^j \vrk(C_j).$$
When  $\vrk(C_j)< \infty$ for some $j \geq 0$ and $\Gamma$ is sofic, we define the {\it $j$-th mean rank of $\cC_\ast$} as
$$\mrk_j(\cC_\ast)=\mrks (\Coker \partial_{j+1})-\mrks(\Ima \partial_j|C_{j-1}).$$
If furthermore $C_j$ and $C_{j-1}$ are countable, we define the {\it j-th mean topological dimension of $\cC_\ast$} as
$$\mdim_j(\widehat{\cC_\ast}):=\mdims(\Ker \widehat{\partial_{j+1}})-\mdims(\Ima \widehat{\partial_j}|\widehat{C_{j-1}}).$$
\end{definition}

\begin{remark}
\begin{enumerate}
\item  When $\cC_\ast$ is a chain complex  of $\CG$-modules, since $\CG$ is flat as a $\ZG$-module, we have $\beta_j^{(2)}(\cC_\ast)=\dim_\LG H_j(\LG \otimes_\CG \cC_\ast)$, which extends the definition of $L^2$-Betti numbers for chain complexes of $\CG$-modules \cite[Definition1.16, Theorem 6.24]{Lueck02B}.

\item By Theorem \ref{addition formula for mrk}, we know that the $j$-th mean rank and $j$-th mean dimension are well defined.

\item Wall gave some criteria when a chain complex of $\ZG$-modules can be realized as the chain complex of a $\Gamma$-CW complex \cite[Theorem 2]{Wall66}.
\end{enumerate}
\end{remark}

Let $\cM$ be a $\ZG$-module. A projective resolution of $\cM$ is an exact sequence of $\ZG$-modules
$$\cdots \stackrel{\partial_2}{\longrightarrow} C_1 \stackrel{\partial_1}{\longrightarrow} C_0 \stackrel{\partial_0}{\longrightarrow} \cM \longrightarrow  0$$
in which each $C_j$ is a projective $\ZG$-module. Denote by $\cC_\ast$ its deleted projective resolution
$$\cdots \stackrel{\partial_2}{\longrightarrow} C_1 \stackrel{\partial_1}{\longrightarrow} C_0  \longrightarrow  0,$$
 which is a chain complex of $\ZG$-modules. We similarly have the notion of free resolution. Apply the notion of free module, we know that any $\ZG$-module admits a free resolution \cite[Proposition 10.32]{Rotman10B}.

\begin{definition} \label{j-th mdim}

 For each $j\geq 0$,  we define the {\it j-th $L^2$-Betti number of $\cM$}  as 
$$\beta_j^{(2)}(\cM):=\beta_j^{(2)}(\cC_\ast).$$
The {\it j-th $L^2$-Betti number $\beta_j^{(2)}(\Gamma)$   of  $\Gamma$} is defined as the j-th $L^2$-Betti number of the trivial $\ZG$-module $\bZ$.

If $\vrk(C_j) < \infty$ for all $j \geq 1$ and $C_j=0$ as $j $ is large enough, we define the {\it Euler characteristic} of $\cM$ as 
$$\chi(\cM):=\chi(\cC_\ast).$$
When $\vrk(C_j)< \infty$ for some $j \geq 0$ and $\Gamma$ is sofic, we define the {\it $j$-th mean rank of $\cM$} as
$$\mrk_j(\cM):=\mrk_j(\cC_\ast).$$
If furthermore $\cM$ is countable, we can choose $\cC_\ast$ such that each $C_j$ for $j \geq 0$ is countable and  define the {\it j-th mean topological dimension of $\cM$} as
$$\mdim_j(\widehat{\cM}):=\mdim_j(\widehat{\cC_\ast}).$$

\end{definition}

\begin{remark}
In fact, $\beta_j^{(2)}(\cM)$ is the von Neumann-L\"{u}ck dimension of ${\rm tor}_j^{\ZG}(\LG, \cM)$ (see  \cite[Page 836]{Rotman10B} for definition).
Based on the Comparison Theorem for projective resolutions \cite[Theorem 10.46]{Rotman10B},  any two projective resolutions of $\cM$ are homotopy equivalent,  we know that $\beta_j^{(2)}(\cM)$ does not depend on the choice of projective resolutions \cite[Corollary 10.51]{Rotman10B}. We refer the reader to \cite[Chapter VIII]{Brown94B} for discussions on when a $\ZG$-module admits a ``small" projective resolution.
\end{remark}

\begin{proposition} \label{0th case}
For any $\ZG$-module $\cM$, we have $\beta_0^{(2)}(\cM)=\vrk(\cM)$ and $\mrk_0(\cM)=\mrks(\cM)$. When $\cM$ is countable, we have $\mdim_0(\hcM)=\mdims(\hcM)$.
\end{proposition}

\begin{proof} 
From the exactness, we have
$$\LG \otimes C_0/\Ima 1\otimes \partial_1=\LG \otimes C_0/\Ker 1\otimes \partial_0\cong \Ima 1\otimes \partial_0=\LG \otimes \cM$$
and 
$$C_0/\Ima \partial_1=C_0/\Ker \partial_0 \cong \Ima \partial_0 =\cM.$$
So by definition, we have
$$\beta_0^{(2)}(\cM)=\dim_\LG (\LG \otimes C_0/\Ima 1\otimes \partial_1)=\dim_\LG (\LG \otimes \cM)=\vrk(\cM)$$
and 
$$\mrk_0(\cM)=\mrks(\Coker \partial_1)=\mrks(C_0/\Ima \partial_1)= \mrks(\cM).$$

From the exactness, we have $\Ker \widehat{\partial_1}=\Ima \widehat{\partial_0}\cong \hcM$. So when $\cM$ is countable, we have 
$$\mdim_0(\hcM)=\mdims(\Ker \widehat{\partial_1})=\mdims(\hcM).$$

\end{proof}

\begin{example}
Let $\Gamma=\bF_2$ be the free group with generators $a$ and $b$. Set $f=(a-1, b-1)^T \in (\bZ \Gamma)^{2 \times 1}$.  Then $\cM:=(\ZG)^{1\times 1}/(\ZG)^{1\times 2}f$ has the following free resolution:
$$0 \to (\ZG)^{1\times 2} \stackrel{R (f)}{\longrightarrow} (\ZG)^{1\times 1} \to \cM \to 0,$$
where $R(f)$ sends $x$ to $xf$.
Note that $\cM \cong \bZ$ as the $\ZG$-modules for the trivial $\ZG$-module $\bZ$. By  \cite[Lemma 6.36]{Lueck02B}, we know $\LG\otimes_{\ZG} \bZ =0 $. Thus $\beta_0^{(2)}(\cM)=0$. It follows that $\beta_1^{(2)}(\cM)=\dim_\LG \Ker 1\otimes R(f)=\beta_0^{(2)}(\cM)-\chi(\cM)=0-(1-2)=1$. This example is essentially the same as $0 \to \Ker \ \varepsilon \to \bZ \Gamma \stackrel{\varepsilon}{\to } \bZ \to 0$, where $\varepsilon$ is the argumentation map. See \cite[Chapter IV, Theorem 2.12]{Dicks80} for a characterization of when $\Ker  \varepsilon$ is a projective $\ZG$-module.  

\end{example}

%%%%%%%%%%%%%%%%%%%%%%%%%%%%%%%%%%%%%%%%%%%%%%%%%%%%%%%%%%%%%%%%%%%%%%%%%%%%%%%%%%%%%%%%%%
\section{Main Results} \label{S-mean length}

Corollary \ref{main theorem} follows when we apply Theorem \ref{theorem for correspondence}  to a deleted projective resolution of a $\ZG$-module $\cM$.  

\begin{proof}[Proof of Theorem \ref{theorem for correspondence}]
(1) Note that for any $\ZG$-module homomorphism $\varphi: \cM \to \cN$ and the inclusion map $i: \Ima \varphi \to \cN$, we have $\Ima 1\otimes \varphi=\Ima 1\otimes i$. Thus $\dim_{\LG}(\Ima 1\otimes \varphi)=\vrk(\Ima \varphi|\cN)$. For $\ZG$-modules $\Ima \partial_{j+1} \subseteq C_j$, we have 
$$\vrk(C_j)=\vrk(\Ima \partial_{j+1}|C_j)+\vrk(\Coker \partial_{j+1}).$$
Since the function $\dim_\LG(\cdot)$ is additive, we have
\begin{align*}
\beta_j^{(2)}(\cC_\ast)&=\dim_\LG (\Ker 1\otimes \partial_j) -\dim_\LG (\Ima1\otimes \partial_{j+1})\\
&=(\dim_\LG (\LG\otimes C_j) -\dim_\LG (\Ima1\otimes \partial_j))-\dim_\LG (\Ima1\otimes \partial_{j+1})\\
&=(\vrk(\Ima \partial_{j+1}|C_j)+\vrk(\Coker \partial_{j+1})) -\dim_\LG (\Ima1\otimes \partial_j)-\dim_\LG (\Ima1\otimes \partial_{j+1})\\
&=\vrk (\Coker \partial_{j+1})-\vrk(\Ima \partial_j|C_{j-1}).
\end{align*}

(2) For any subgroup $H$ of a discrete abelian group $G$, denote by $H^\perp$ the elements of $\widehat{G}$ which vanish on $H$. By Pontryagin duality,  we have $H^\perp \cong \widehat{G/H}$. 
Thus 
$$\Ker \widehat{\partial_{j+1}}=\{\chi \in \widehat{C_j}: \chi \circ \partial_{j+1}=0\}=(\Ima \partial_{j+1})^\perp \cong \widehat{C_j/\Ima \partial_{j+1}}=\widehat{\Coker \partial_{j+1}}.$$
By definition, we have $\mdims(\Ima \widehat{\partial_j}|\widehat{C_{j-1}})=\mdims(\widehat{\Ima \partial_j}|\widehat{C_{j-1}})$.
Thus by Theorem \ref{addition formula for mrk}, the equalities follow from (1).

(3) By the addition formula of von Neumann-L\"{u}ck dimension, we first have
$$\vrk(C_{j-1})=\vrk(\Ima \partial_j|C_{j-1})+\vrk(\Coker \partial_j).$$
When $\cC_\ast$ is exact at $C_j$, we have $\Coker \partial_{j+1}=\Ima \partial_j$. Thus the first statement follows from (1). The second statement follows from Pontryagin duality and the first statement.
\end{proof}

The follow proposition is an immediate consequence of Theorem \ref{theorem for correspondence}, which can also be proved directly.
\begin{proposition} \label{Euler equality}
Let $\cC_\ast$ be a chain complex of $\ZG$-modules: $0 \to C_k \to \cdots \to C_0 \to 0$ for some $k \in \bN$. If $\vrk(C_j) < \infty$ for all $j \geq 1$, then
$$\sum_{0 \leq j \leq k} (-1)^j \beta_j^{(2)}(\cC_\ast)=\chi(\cC_\ast)=\sum_{0 \leq j \leq k} (-1)^j \mrk_j(\cC_\ast).$$
In particular, when $\cC_\ast$ is exact, we have
$$\chi(\cC_\ast)=\sum_{0 \leq j \leq k} (-1)^j (\mrks(\Ima \partial_j)-\mrks(\Ima \partial_j|C_{j-1})).$$
\end{proposition}

The following lemma reduces the Juzvinski\u{\i} formula for mean rank to the finitely generated case.
\begin{lemma} \label{relativity}
Suppose that $\mrks(\cM_1|\cM_2)=\mrks(\cM_1)$ holds when $\cM_2$ is a finitely generated free $\ZG$-module and $\cM_1$ is a finitely generated $\ZG$-submodule of $\cM_2$. Then $\mrks(\cM_1|\cM_2)=\mrks(\cM_1)$ holds for any $\ZG$-modules $\cM_1 \subseteq \cM_2$. In particular, $\Gamma$ satisfies Juzvinski\u{\i} formula for mean rank if and only if $\mrks(\cM_1|\cM_2)=\mrks(\cM_1)$ holds when $\cM_2$ is a finitely generated free $\ZG$-module and $\cM_1$ is a finitely generated $\ZG$-submodule of $\cM_2$. 
\end{lemma}

\begin{proof}
{\bf Case 1.} Both $\cM_1$ and $\cM_2$ are finitely generated.

Write $\cM_2$ as $\cM_2=(\ZG)^n/\cN$ and $\cM_1$ as $\cM_1=\cM_1'/\cN$ for some $n \in \bN$ and some $\ZG$-submodules $\cN \subseteq \cM_1' \subseteq (\ZG)^n$. Write $\cN$ as the union of an increasing net of finitely generated submodules $\{\cN_j\}_{j \in \cJ}$ of $\cN$. Note that for each $j$, by Proposition \ref{properties of mrk}, we have
$$\mrks(\cN_j)=\mrks(\cN_j|\cN_j) \geq \mrks(\cN_j|\cM_1') \geq \mrks(\cN_j|(\ZG)^n)=\mrks(\cN_j).$$
Thus $\mrks(\cN_j|(\ZG)^n)=\mrks(\cN_j|\cM_1')$ for all $j$. Moreover, by Proposition \ref{properties of mrk},
$$\mrks(\cN|(\ZG)^n)=\sup_{j \in \cJ} \mrks(\cN_j|(\ZG)^n)=\sup_{j \in \cJ} \mrks(\cN_j|\cM_1')=\mrks(\cN|\cM_1').$$
Then by Theorem \ref{addition formula for mrk} and Proposition \ref{properties of mrk}, we have
\begin{align*}
    &\mrks(\cM_1|\cM_2)\\
    &=\mrks((\ZG)^n/\cN) -\mrks((\ZG)^n/\cM_1')\\
    &=(n-\mrks(\cN|(\ZG)^n))-(n-\mrks(\cM_1'|(\ZG)^n))\\
    &=\mrks(\cM_1')-\mrks(\cN|\cM_1')\\
    &=\mrks(\cM_1'/\cN)=\mrks(\cM_1).
\end{align*}

{\bf Case 2.} $\cM_1$ is finitely generated.

Write $\cM_2$ as the union of an increasing net of finitely generated submodules $\{\cM_j'\}_{j \in \cJ}$ of $\cM_2$. By  Proposition \ref{properties of mrk} and the conclusion of Case 1, we have
$$\mrks(\cM_1|\cM_2)=\inf_{j \in \cJ} \mrks(\cM_1|\cM_j')=\inf_{j \in \cJ} \mrks(\cM_1)=\mrks(\cM_1).$$

Now we consider the general case. Write $\cM_1$ as the union of an increasing net of finitely generated submodules $\{\cM_j'\}_{j \in \cJ}$ of $\cM_1$. Applying Proposition \ref{properties of mrk} and the conclusion of Case 2, we have
$$\mrks(\cM_1|\cM_2)=\sup_{j \in \cJ} \mrks(\cM_j'|\cM_2)=\sup_{j \in \cJ} \mrks(\cM_j'|\cM_1)=\mrks(\cM_1).$$

Suppose that  $\Gamma$ satisfies Juzvinski\u{\i} formula for mean rank. Let $\cM_2$ be a finitely generated free $\ZG$-module and $\cM_1$ be finitely generated $\ZG$-submodule of $\cM_2$. Consider the  quotient map $\cM_2 \to \cM_2/\cM_1$. Then the conclusion follows immediately by Theorem \ref{addition formula for mrk}. The converse direction also follows immediately by Theorem \ref{addition formula for mrk}.
\end{proof}

\begin{remark} \label{L}
Let $\rL$ be a function as in \cite[Lemma 7.7]{LL15A} satisfying all the properties (i)-(v). Then the corresponding result in Lemma \ref{relativity} still holds without change of proof. In particular, the corresponding statement holds for von Neumann-L\"{u}ck rank.
\end{remark}

\begin{proof}[Proof of Corollary \ref{dimension-flat}]
The second statement follows from Theorem \ref{addition formula for mrk} and the first statement. Suppose $\Gamma$ satisfies L\"{u}ck's dimension-flatness, in particular, for any finitely presented $\ZG$-module $\cM$, we have 
$\beta_1^{(2)}(\cM)=0$. Write $\cM$ as $\cM =(\ZG)^n/(\ZG)^mf$ for some $f \in M_{m, n}(\ZG)$. Then a projective resolution of $\cM$ can be 
$$\cdots \stackrel{\partial_2}{\longrightarrow} (\ZG)^m \stackrel{\partial_1}{\longrightarrow} (\ZG)^n \to \cM \to 0,$$
where $\partial_1=R(f)$.
Then by Theorem \ref{theorem for correspondence}, we have
\begin{align*}
    &\vrk(\Ima \partial_1)-\vrk(\Ima \partial_1|(\ZG)^n)\\
    &=\vrk(\Coker \partial_2)-\vrk(\Ima \partial_1|(\ZG)^n)\\
    &=\beta_1^{(2)}(\cM)=0.
\end{align*}
By Remark \ref{L}, we have $\Gamma$ satisfies Juzvinski\u{\i} formula for von Neumann-L\"{u}ck rank.

For the ``if" part, by Theorem \ref{theorem for correspondence}, we first have $\beta_1^{(2)}(\cM)=0$ for any finitely presented $\ZG$-module $\cM$.  Since both $\dim_\LG(\cdot)$ and ${\rm tor_j^{\ZG}(\LG, \cdot)}$  commutes with the colimits \cite[Theorem 6.7]{Lueck02B} \cite[Proposition 10.99]{Rotman10B}, we have $\beta_1^{(2)}(\cM)=0$ for any $\ZG$-module $\cM$. 

 Let $0 \to \cN \to \cF \to \cM \to 0$ be an exact sequence of $\ZG$-modules such  that $\cF$ is free. Then a projective resolution of $\cN$ induces a projective resolution of $\cM$. So $\beta_2^{(2)}(\cM)=\beta_1^{(2)}(\cN)=0$. Inductively we have $\beta_j^{(2)}(\cM)=0$ for all $j \geq 1$  and $\ZG$-module $\cM$.

\end{proof}

The following proposition was implicitly proven in \cite[Conjecture 6.48]{Lueck02B}. For convenience, we give a proof. 

\begin{proposition} \label{subgroup}
Let $\Gamma$ be a discrete group (not necessarily sofic) and $H$ be a subgroup of $\Gamma$. Then for any $\bZ H$-module $\cM$, we have $\beta_j^{(2)}(\ZG \otimes_{\bZ H}\cM)=\beta_j^{(2)}(\cM)$. In particular, taking subgroups respects the property of L\"{u}ck's dimension-flatness.
\end{proposition}

\begin{proof}
 Let $\cC_\ast \to \cM$ be a free resolution of $\cM$.  Since $\ZG$ is a flat $\bZ H$-module, we get a free resolution $\ZG \otimes_{\bZ H} \cC_\ast \to \ZG \otimes_{\bZ H} \cM$ of the $\ZG$-module $\ZG \otimes_{\bZ H} \cM$. Since the induction functor $\LG \otimes_{\cL H} \cdot$ is flat \cite[Theorem 6.29 (1)]{Lueck02B}, we have
$$\LG \otimes_{\cL H} H_j(\cL H \otimes_{\bZ H}\cC_\ast) \cong H_j(\LG \otimes_{\cL H}(\cL H \otimes_{\bZ H} \cC_\ast))\cong H_j(\LG \otimes_{\ZG}(\ZG \otimes_{\bZ H}\cC_\ast)).$$
Thus by \cite[Theorem 6.29 (2)]{Lueck02B}, we have
\begin{align*}
    &\beta_j^{(2)}(\ZG \otimes_{\bZ H}\cM)\\
    &=\dim_{\LG} H_j(\LG \otimes_{\ZG}(\ZG \otimes_{\bZ H} \cC_\ast))\\
    &= \dim_\LG (\LG \otimes_{\cL H} H_j(\cL H \otimes_{\bZ H} \cC_\ast))\\
    &=\dim_{\cL H} H_j(\cL H \otimes_{\bZ H}\cC_\ast)=\beta_j^{(2)}(\cM).
\end{align*}

\end{proof}

As a consequence of Corollary \ref{dimension-flat} and Proposition \ref{subgroup}, we have:
\begin{corollary} \label{subgroup failure}
If a subgroup $H$ of a sofic group $\Gamma$ violoates L\"{u}ck's dimension-flatness, then $\Gamma$ violoates Juzvinski\u{\i} formula for mean dimension. In particular, if $\beta_j^{(2)}(H) > 0$ for some $j\geq 1$, then $\Gamma$  violoates Juzvinski\u{\i} formula for mean rank  and mean dimension.
\end{corollary}

We refer the reader to \cite[Section 5]{PT11} for some discussions on $L^2$-Betti numbers of subgroups. L\"{u}ck's dimension-flatness for amenable groups \cite[Theorem 6.37]{Lueck02B} can be interpreted in terms of relative sofic mean rank.

\begin{corollary} \label{amenable case}
Amenable groups satisfy L\"{u}ck's dimension-flatness.
\end{corollary}

\begin{proof}
By \cite[Theorem 5.1]{LL15A}, we know $\mrks(\cM_1|\cM_2)=\mrks(\cM_1)$ holds for any $\ZG$-modules $\cM_1 \subseteq \cM_2$. By Theorem \ref{addition formula for mrk}, $\Gamma$ satisfies Juzvinski\u{\i} formula for mean rank. Thus by Corollary \ref{dimension-flat}, $\Gamma$ satisfies L\"{u}ck's dimension-flatness.
\end{proof}

\begin{corollary}
Let $\Gamma$ be a residually finite group and $\{\Gamma_i\}_i$ be a  sequence of finite-indexed decreasing normal subgroups of $\Gamma$ with the intersection $\{e_\Gamma\}$. Let $\cC_\ast$ be a chain complex of finitely generated free $\ZG$-modules. Then
$$\mrk_j(\cC_\ast)=\lim_{i\to \infty} \frac{\rk(H_j(\Gamma_i \setminus \cC_\ast))}{|\Gamma/\Gamma_i|}.$$

\end{corollary}

\begin{proof}
Since residually finite groups satisfy L\"{u}ck's approximation formula for $L^2$-Betti numbers \cite[Theorem 0.1]{Lueck94}, by the similar argument as in \cite[Lemma 13.4]{Lueck02B}, we have
$$\beta_j^{(2)}(\cC_\ast)=\lim_{i\to \infty} \frac{\rk(H_j(\Gamma_i \setminus \cC_\ast))}{|\Gamma/\Gamma_i|}.$$
Then the conclusion follows from Theorem \ref{theorem for correspondence}.
\end{proof}

%%%%%%%%%%%%%%%%%%%%%%%%%%%%%%%%%%%%%%%%%%%%%%%%%%%%%%%%%%%%%%%%%%%%%%%%%%%%%%%%%%%%%%%%%%%%%%%%%%%%%%%%%%%%%%%%%%%%%%%%%%%

\end{document}